\documentclass{amsart}
\usepackage{graphicx}
\usepackage{graphicx,esint}
\usepackage{amssymb,amscd,amsthm,amsxtra}
\usepackage{latexsym}
\usepackage{epsfig}
\usepackage{graphicx,color}

\newtheorem{thm}{Theorem}[section]
\newtheorem{cor}[thm]{Corollary}
\newtheorem{lem}[thm]{Lemma}
\newtheorem{prop}[thm]{Proposition}
\theoremstyle{definition}
\newtheorem{defn}[thm]{Definition}
\theoremstyle{remark}
\newtheorem{rem}[thm]{Remark}
\numberwithin{equation}{section}

\newcommand{\R}{\mathbb R}
\newcommand{\be}{\begin{equation}}
\newcommand{\ee}{\end{equation}}

\newcommand{\eps}{\varepsilon}

\newcommand{\p}{\partial}

\newcommand{\comment}[1]{}

\begin{document}
\title[ Existence theory]{Perron's solutions for two-phase
free boundary problems with distributed sources}
\author{Daniela De Silva}
\address{Department of Mathematics, Barnard College, Columbia University,
New York, NY 10027}
\email{\texttt{desilva@math.columbia.edu}}
\author{Fausto Ferrari}
\address{Dipartimento di Matematica dell' Universit\`a, Piazza di Porta S.
Donato, 5, 40126 Bologna, Italy.}
\email{\texttt{fausto.ferrari@unibo.it}}
\author{Sandro Salsa}
\address{Dipartimento di Matematica del Politecnico, Piazza Leonardo da
Vinci, 32, 20133 Milano, Italy.}
\email{\texttt{sandro.salsa@polimi.it }}
\thanks{D.D. is partially supported by NSF grant DMS-1301535.  F.~ F.~ is supported by the ERC starting grant project 2011 EPSILON (Elliptic PDEs and Symmetry of Interfaces and Layers for Odd Nonlinearities)  277749 and  by RFO grant, Universit\`a di Bologna, Italy.   
}

\begin{abstract}
We use Perron method to construct a weak solution to a two-phase free boundary problem with right-hand-side. We thus extend the results in \cite{C3} for the homogeneous case.  
\end{abstract}

\maketitle


\section{Introduction}

Let $\Omega $ be a bounded Lipschitz domain in $\mathbb{R}^{n}$ and let $%
A=A(x)$ be a symmetric matrix with H\"{o}lder continuous coefficients in $%
\Omega ,$ which is uniformly elliptic, i.e. 
\begin{equation*}
\lambda \mid \xi \mid ^{2}\leq \sum_{i,j=1}^{n}a_{ij}(x)\xi _{i}\xi _{j}\leq
\Lambda \mid \xi \mid ^{2},\quad \forall x\in \Omega ,\quad \xi \in \mathbb{R%
}^{n}
\end{equation*}%
for some $0<\lambda \leq \Lambda .$ Denote by $$\mathcal L := \text{div}(A(x)\nabla \cdot).$$

Let $f_{1},f_{2}\in L^{\infty }(\Omega )$. We consider the following
two-phase inhomogeneous free boundary problem (f.b.p. in the sequel)%
\begin{equation}
\left\{ 
\begin{array}{ll}
\mathcal L u= f_{1}, & \text{ \ in }\Omega ^{+}(u)=\{u>0\} \\ 
\mathcal L u=f_{2}\chi _{\{u<0\}} & \text{ \ in }\Omega
^{-}(u)=\{u\leq 0\}^{\circ} \\ 
u_{\nu }^{+}=G(u_{\nu }^{-},x,\nu ) & \text{ \ on }F(u)=\partial \{u>0\}\cap
\Omega .%
\end{array}%
\right.  \label{FBintro}
\end{equation}

Here $\nu =\nu (x)$ denotes the unit normal to $F(u)$ at $x$, pointing
towards $\Omega ^{+}(u).$ The function $G(\beta ,x,\nu )$ is strictly
increasing in $\beta $, Lipschitz continuous in all its arguments and $%
G(0):=\inf_{x\in \Omega ,\left\vert \nu \right\vert=1 }G(0,x,\nu ) >0.$
Conormal derivatives $\nabla u^{\pm }\cdot \nu $ can be equally considered
instead of normal derivatives.

Problems of this kind arise in several  contexts, see \cite{DFSs1} for a list. 

In this paper, our main purpose is to construct a weak solution assuming given boundary
data, via Perron method, extending the results of the seminal paper \cite{C3} in
the homogeneous case. Before stating our main result, we give the definition
of weak solution of problem (\ref{FBintro}).

Given a continuous function $v$ on $\Omega$, we say that a point $x_{0}\in F(v)$ is regular from the right (resp. left)
if there is a ball $B\subset \Omega ^{+}(v)$ (resp. $B\subset\Omega ^{-}(v)$), such
that $\overline{B}\cap F(v)=\{x_{0}\}$. In what follows, $\nu=\nu(x_0)$ represents the unit normal to $\p B$ at $x_0$ pointing toward $\Omega^+(v).$

\begin{defn}
\label{visc} We say that $u\in C(\Omega )$ is a weak solution of the f.b.p. (\ref{FBintro})
if:
\smallskip

a) $\mathcal L u=f_{1}$ in $\Omega ^{+}(u)$ and $\mathcal L u=f_{2}\chi _{\{u<0\}}$\ in $\Omega
^{-}(u),$ in the usual weak sense;\smallskip

b) $u$ satisfies the free boundary condition in \eqref{FBintro} in the
following sense:\smallskip

\quad (i) If $x_{0}\in F(u)$ is regular from the right with touching ball $B$
then
\begin{equation*}
u^{+}(x)\geq \alpha \langle x-x_{0},\nu \rangle ^{+}+o(|x-x_{0}|)\quad \text{%
in $B$, with $\alpha \geq 0$}
\end{equation*}

and
\begin{equation*}
u^{-}(x)\leq \beta \langle x-x_{0},\nu \rangle ^{-}+o(|x-x_{0}|)\quad \text{%
in $B^{c}$, with $\beta \geq 0$}
\end{equation*}

with equality along every non-tangential domain, and
\begin{equation*}
\alpha \leq G(\beta ,x_{0},\nu \left( x_{0}\right) ).
\end{equation*}

\quad (ii) If $x_{0}\in F(u)$ is regular from the left with touching ball $B,$
then
\begin{equation*}
u^{-}(x)\geq \beta \langle x-x_{0},\nu \rangle ^{-}+o(|x-x_{0}|)\quad \text{%
in $B$, with $\beta \geq 0$}
\end{equation*}%
\begin{equation*}
u^{+}(x)\leq \alpha \langle x-x_{0},\nu \rangle ^{+}+o(|x-x_{0}|)\quad \text{%
in $B^{c}$, with $\alpha \geq 0$}
\end{equation*}

with equality along every non-tangential domain, and
\begin{equation*}
\alpha \geq G(\beta ,x_{0},\nu \left( x_{0}\right) ).
\end{equation*}
\end{defn}

Note that (i) (resp. (ii)) expresses a
supersolution (resp. subsolution) condition at points regular from the rigth
(resp. left). While this definition slightly differs from the one in \cite{C3}, it
is indeed equivalent to it (see (\cite{CS}).

Our solution is constructed as the infimum over an admissible class of
supersolutions $\mathcal{F}$. 

\begin{defn}
\label{admis}A function $w\in \mathcal{F}$ if $w\in C(%
\overline{\Omega })$ and if

\begin{itemize}
\item[(a)] $w$ is a weak solution to 
\begin{equation*}
\mathcal L w \leq f_{1}  \text{ \ in $%
\Omega ^{+}(w)$} \quad  \text{and} \quad 
\mathcal L w\leq f_{2}\chi _{\{w<0\}} \quad  \text{in 
$\Omega ^{-}(w).$}%
\end{equation*}

\item[(b)] If $x_{0}\in F(u)$ is regular from the left, then near $x_{0},$ 
\begin{equation*}
w^{+}\leq \alpha \langle x-x_{0},\nu \left( x_{0}\right) \rangle ^{+}+o(\mid
x-x_{0}\mid ),\quad \alpha \geq 0,
\end{equation*}%
\begin{equation*}
w^{-}\geq \beta \langle x-x_{0},\nu \left( x_{0}\right) \rangle ^{-}+o(\mid
x-x_{0}\mid ),\quad \beta \geq 0,
\end{equation*}%
with 
\begin{equation*}
\alpha <G(\beta ,x_{0},\nu \left( x_0\right) ).
\end{equation*}

\item[(c)] If $x_{0}\in F(w)$ is not regular from the left, then near $x_{0}$%
, 
\begin{equation*}
w(x)=o(|x-x_{0}|).
\end{equation*}
\end{itemize}
\end{defn}

We also need to introduce a
minorant subsolution. We say that a locally Lipschitz function $%
\underline{u},$ defined in $\Omega ,$ is a \emph{minorant} if:

a) $\underline{u}$ is a weak solution to 
\begin{equation*}
\mathcal L \underline{u} \geq f_{1} \quad 
\text{ \ in $\Omega ^{+}(\underline{u})$} \quad \text{and} \quad 
\mathcal L \underline{u} \geq f_{2}\chi _{\{%
\underline{u}<0\}}  \text{ \ in $\Omega ^{-}(\underline{u}).$}%
\end{equation*}

b) Every $x_{0}\in F(\underline u)$ is regular from the right and near $x_{0},$ 
\begin{equation*}
\underline{u}^{-}\leq \beta \langle x-x_{0},\nu \left( x_{0}\right) \rangle
^{-}+o(\mid x-x_{0}\mid ),
\end{equation*}%
\begin{equation*}
\underline{u}^{+}\geq \alpha \langle x-x_{0},\nu \left( x_{0}\right) \rangle
^{+}+o(\mid x-x_{0}\mid ),
\end{equation*}%
with 
\begin{equation*}
\alpha >G(\beta ,x_{0},\nu \left( x_{0}\right) ).
\end{equation*}

We are now ready to state our main result.

\begin{thm}
\label{teoprandbach} Let $\phi $ be a continuous function on $\partial
\Omega $ and $\underline{u}$ be a minorant of our free boundary problem,
with boundary data $\phi .$ Then
\begin{equation*}
u=\inf \{w: w\in \mathcal{F}, \ w\geq \underline{u}\ \text{in $\overline{%
\Omega }$}\}
\end{equation*}%
is a solution to $(\ref{FBintro})$ such that $u=\phi $ on $\partial \Omega $%
, as long as the set on the right is non-empty.
\end{thm}

Concerning the regularity of the free boundary, we prove the following
standard result.
\begin{thm}
\label{regularity}The free boundary $F(u)$ has finite $(n-1)$-dimensional
Hausdorff measure. More precisely, there exists a universal constant $r_{0}>0
$ such that for every $r<r_{0},$ for every $x_{0}\in F(u),$ 
\begin{equation*}
\mathcal{H}^{n-1}(F(u)\cap B_{r}(x_{0}))\leq r^{n-1}.
\end{equation*}%
Moreover, denoting with $F^*(u)$ the reduced free boundary
\begin{equation*}
\mathcal{H}^{n-1}(F(u)\setminus F^{\ast }(u))=0.
\end{equation*}
\end{thm}

In a forthcoming paper we shall adress further regularity properties of the
free boundary. In particular, compactness properties of the minimal
solutions constructed in Theorem \ref{teoprandbach} and the flatness result
in \cite{DFS4} will imply the following corollary, new even in the homogeneous
case.

\begin{thm}
$F(u)$ is a $C^{1,\gamma }$ surface in a neighborhood of $H^{n-1}$ a.e.
point $x_{0}\in F(u).$
\end{thm}

The paper follows the main guidelines of \cite{C3}, although the presence of a
distributed source requires to face new situations and requires new delicate
arguments, especially in Sections \ref{uLipschitz} and \ref{ondegeneracyu+}. The organization is as follows.
In Section \ref{preliminaries1} we prove some preliminary lemmas frequently used throughout the paper. In
Section \ref{lipschitzregu+}
 we prove that $u^{+}$ is Lipschitz continuous. Then in Section \ref{uLipschitz}
we show that $u$ is Lipschitz continuous and it satisfies the equation in
both $\Omega ^{+}\left( u\right) $ and $\Omega ^{-}\left( u\right) $. Linear
growth near the free boundary and the non-degeneracy of $u^+$ are proved in
Section \ref{ondegeneracyu+}. The following section, Section \ref{supersolutionuf}, is devoted to the proof that $u$
satisfies the free boundary condition in the supersolution sense (part b(i).
in Definition \ref{visc}). Finally in Section \ref{subsolutionuf} we prove that $u$ satisfies
the free boundary condition in the subsolution sense (part b(ii). in Definition %
\ref{visc}) and hence it is a weak solution to our problem. We conclude our
paper with the regularity result in Theorem \ref{regularity} in Section \ref{sizereducedbound}.

Throughout the paper, constants depending possibly only on $[A]_{C^{0,\gamma
}},$ $\Vert f_{1}\Vert _{\infty }$,$\Vert f_{2}\Vert _{\infty },$ $\lambda
$, $\Lambda$, $G(0)$, $ n$ are called universal. Whenever a constant depends on
other parameters, that dependance will be explicitly noted. Finally, for standard regularity theory for weak solution to divergence form equations, we refer the reader to \cite{GT}.
\section{Preliminaries} \label{preliminaries1}

In this section we introduce some notation and prove some useful lemmas which will be used several times in the paper.  

{\bf Notation.} As usual, $B_r(x_0)$ denotes the ball in $\R^n$ of radius $r$ and center $x_0$. When $x_0=0$ we omit the dependence on $x_0.$ Also, throughout the paper we will use the following notation for rescalings of size $r$ around $x_0$:
\begin{align}\label{notation}&u_r(x):= \frac{u(x_0 + r x)}{r}, \quad x \in  \Omega_r := \frac{\{x-x_0 : x\in \Omega\}}{r},\\
\nonumber &A_r(x)= A(x_0 + r x), \quad f^r(x)= f(x_0 + r x), \quad \mathcal L_r = \text{div}(A_r(x)\nabla \cdot)\\
\nonumber &G_r(\alpha, x, \nu)= G(\alpha, x_0+rx, \nu(x_0+r x)).\end{align}

Finally, we denote
$$G(0):=\inf_{x\in \Omega ,\left\vert \nu \right\vert=1 }G(0,x,\nu )>0.$$

\begin{lem}\label{barrier} Let $v \in C(\overline{B_r(x_0)}),$ $v\geq 0$ for $r \leq 1.$ Assume that 
$$\mathcal L v = f  \in L^\infty \quad \text{in $B_r(x_0)$}$$ and $$v(y_0)=0, \quad y_0 \in \p B_r(x_0).$$ Denote by $\nu$ the inner unit normal to $\p B_r(x_0)$ at $y_0.$
Then,
$$v(x) \geq \alpha \langle x-y_0, \nu \rangle^+ + o(|x-y_0|)$$ with $$\alpha \geq \bar c \frac{v(x_0)}{r} - \bar C r \|f\|_\infty$$ and $\bar c, \bar C>0$ depending on $[A]_{0,\gamma}, \lambda, \Lambda, n.$
\end{lem}
\begin{proof}
Let $$v_r(x) = \frac{v(x_0 + rx)}{r}, \quad x \in B_1.$$ Then,
$$v_r \geq 0 \quad \text{in $B_1$,} \quad v_r(y_r)= 0, \quad y_r \in \p B_1$$ and
$$\mathcal L_r v_r = r f^r \quad \text{in $B_1$.}$$
Notice that  $A_r(x) = A(x_0+rx)$ has the same ellipticity constants as $A$ and its $C^{0,\gamma}$ norm is controlled by the $C^{0,\gamma}$ norm of $A.$

Call $\|f\|_{\infty}= M.$ By Harnack inequality,

$$\inf_{B_{1/2}} v_r \geq c(v_r(0) - r M).$$

Now, denote with $\xi, \eta$ the solutions to the  the following problems:
$$\mathcal L_r \xi =-1 \quad \text{in $B_{1} \setminus \overline{B_{1/2}}$}$$
$$\xi=0 \quad \text{on $\p B_{1/2}$}, \quad \xi=0 \quad \text{on $\p B_{1},$}$$ 
$$\mathcal L_r\eta =0 \quad \text{in $B_{1} \setminus \overline{B_{1/2}}$}$$
$$\eta=1 \quad \text{on $\p B_{1/2}$}, \quad \eta=0 \quad \text{on $\p B_{1}.$}$$ 
Call $$\quad c_1= \xi_\nu|_{\p B_1} >0 \quad c_2 = \eta_\nu|_{\p B_1} >0,$$ with $\nu$ the inner unit normal to $\p B_1.$ Notice that $c_1$ depends only on $[A]_{C^{0,\gamma}}, \lambda, \Lambda, n.$

Define,
$$\phi := c(v_r(0) -rM) \eta +rM \xi \quad \text{in $B_{1} \setminus \overline{B_{1/2}}$.}$$

Then,

$$\mathcal L_r \phi= -rM \geq r f^r \quad \text{in $B_{1} \setminus \overline{B_{1/2}}$,}$$
and 
$$\phi \leq v_r \quad \text{on $\p B_1 \cup \p B_{1/2}.$}$$

Thus,
$$\phi \leq v_ r \quad  \text{in $B_{1} \setminus \overline{B_{1/2}}$,}$$
and hence 
$$v_r (x) \geq (c(v_r(0) -r M)c_1 + r M c_2) \langle x- y_r, \nu \rangle^+ + o(|x-y_r|),$$
which gives the desired result.
\end{proof}

Next we prove the following asymptotic developments lemmas.

\begin{lem}\label{Lemma_1as} Let  $\Omega $ be an open set,  $0\in
\partial \Omega.$ Assume that $B_{\rho }\left( -\rho e^{1}\right)
\subset R^{n}\backslash \overline{\Omega }.$ Let $u$ be a
nonnegative Lipschitz function in $B_{1}\cap \overline{\Omega },$
satisfying $\mathcal{L}u=f$ in $B_{1}\cap \Omega $ and $u=0$  on $\partial \Omega \cap B_{1}.$

Then there exists $\alpha \geq 0$ such that
\begin{equation*}
u\left( x\right) =\alpha x_{1}+o\left( \left\vert x\right\vert \right) \text{
\ \ as }x\rightarrow 0\text{, }x\in \overline{\Omega }\cap B_{1}\text{.}
\end{equation*}%
In particular, if $\alpha >0,$ then along $\partial \Omega
,$
\begin{equation*}
x_{1}=o(\left\vert x\right\vert )\text{ \ \ as }x\rightarrow 0\text{, }x\in
\partial \Omega \cap B_{1}
\end{equation*}%
 that is  $\partial \Omega $ is tangent to the hyperplane $%
x_{1}=0$.
\end{lem}

\begin{proof} We may assume that $\rho <1/3$. We change variables by setting 
\begin{equation*}
y=T\left( x\right) =\frac{e^{1}}{\rho }-\frac{x+\rho e^{1}}{\left\vert
x+\rho e^{1}\right\vert ^{2}}
\end{equation*}%
and define $v\left( y\right) =u\left( T^{-1}\left( y\right) \right) $. Then $%
T\left( 0\right) =0$, and the exterior of the ball $B_{\rho }\left( -\rho
e^{1}\right) $ is mapped onto $B_{1/\rho }\left( e^{1}/\rho \right)
\backslash \left\{ e^{1}/\rho \right\} $. Thus $\Omega ^{\prime }=T(\Omega
)\subset B_{1/\rho }\left( e^{1}/\rho \right) $ and $\Omega ^{\prime }\cap
B_{2}\subset B_{2}^{+}=\left\{ y\in B_{2}:y_{1}>0\right\} $.

Note also that, 
\begin{equation}
y_{1}=\left( \tfrac{2}{\rho ^{2}}-1\right) x_{1}+o\left( \left\vert
x\right\vert \right) .  \label{opiccolo}
\end{equation}%
Moreover, $v$ is Lipschitz in $\overline{\Omega }^{\prime }\cap B_{2}$, $v=0$
on $\partial \Omega ^{\prime }\cap B_{2}$ and%
\begin{equation*}
\mathcal{L}^{\prime }v=\mbox{div}\left( A^{\prime }\left( y\right) \nabla v\right)
=f^{\prime }\left( y\right) \equiv f\left( T^{-1}\left( y\right)\right)\cdot\mid\mbox{det}J\mid  
\text{ \ \ in }\Omega ^{\prime }\cap B_{2}
\end{equation*}%
where $A^{\prime }=JAJ^{\top }\cdot\mid\mbox{det}J\mid$, $J$ being the Jacobian of $T^{-1}$. Note
that if $A$ is symmetric then 
$A'$ is symmetric and 
\begin{equation*}
c\left( \rho ,\lambda \right) I\leq A^{\prime }\left( y\right) \leq C\left(
\rho ,\Lambda \right) I\text{ \ \ \ \ in }\overline{\Omega }^{\prime }\cap
B_{2}.
\end{equation*}%
Extend $v$ by zero in $B_{1}$ outside $\Omega ^{\prime }$. Then (still
calling $v$ the extended function), $\mathcal{L}^{\prime }v\geq -\left\Vert
f'\right\Vert _{\infty }$ in $B_{2}^{+}$ (in a weak sense). We also have $%
v\left( y\right) \leq Cy_{1}$ in $B_{3/2}^{+}$ (compare with the solution of 
$\mathcal{L}^{\prime }z=-\left\Vert f'\right\Vert _{\infty },$ $z=v$ on $\partial
B_{2}^{+}$).

Now, let $w=w\left( x\right) $ be the $\mathcal{L}^{\prime }-$harmonic measure in $%
B_{2}^{+}$ of $S_{2}^{+}=\partial B_{2}\cap \left\{ y_{1}>0\right\} $. $\ $%
Then, by Hopf principle and standard regularity theory,%
\begin{equation}
y_{1}c_{1}\leq w\left( y\right) \leq c_{2}y_{1}\text{ \ \ \ \ \ in }%
\overline{B}_{1}^{+}  \label{harm}
\end{equation}%
with $c_{1},c_{2}$ positive and universal, and, for some universal $\gamma >0
$, 
\begin{equation}
w\left( y\right) =\gamma y_{1}+o\left( \left\vert y\right\vert \right) \text{
as }y\rightarrow 0\text{, }y\in B_{1}^{+}.  \label{asym}
\end{equation}%
Let now for $k\geq 1,$ integer, 
\begin{equation*}
m_{k}=\inf \left\{ m:v\left( y\right) \leq mw\left( y\right) \text{ \ \ \
for every }y\in \overline{\Omega }^{\prime }\cap B_{1/k}\right\} .
\end{equation*}%
Then $\left\{ m_{k}\right\} $ is non increasing and $m_{k}\rightarrow
m_{\infty }\geq 0.$ Moreover 
\begin{equation}
v\left( y\right) \leq m_{\infty }w\left( y\right) +o\left( \left\vert
y\right\vert \right) \text{ \ \ as }y\rightarrow 0\text{, }y\in \overline{%
\Omega }^{\prime }\cap B_{1}.  \label{equality}
\end{equation}%
We claim that equality holds in (\ref{equality}). If not, there exist $%
\delta >0$ and a sequence $\left\{ y_{j}\right\} \in \Omega ^{\prime }\cap
B_{1}$ such that $r_{j}=\left\vert y_{j}\right\vert \rightarrow 0$ and 
\begin{equation*}
v\left( y_{j}\right) \leq m_{\infty }w\left( y_{j}\right) -\delta r_{j}.%
\text{ }
\end{equation*}%
Since both $v$ and $w$ are Lipschitz, we can write%
\begin{equation}
W\left( y\right) \equiv m_{\infty }w\left( y\right) -v\left( y\right) \geq
\delta r_{j}/2\text{ \ \ \ \ \ on }B_{cr_{j}}\left( y_{j}\right) \cap
S_{r_{j}}^{+}  \label{ops}
\end{equation}%
with $c$ depending on $m_{\infty }$ and the Lipschitz constants of $v$ and $w
$ .

On the other hand, (\ref{equality}) implies that 
\begin{equation}
W\left( y\right) \geq -\sigma _{j}r_{j}\text{ \ \ \ \ \ on }S_{r_{j}}^{+}
\label{ops2}
\end{equation}%
with $\sigma _{j}\rightarrow 0$. Rescale by setting%
\begin{equation*}
W_{j}\left( y\right) =m_{\infty }\frac{w\left( r_{j}y\right) }{r_{j}}-\frac{%
v\left( r_{j}y\right) }{r_{j}}=\frac{W\left( r_{j}y\right) }{r_{j}}\text{ \
\ \ \ \ \ }y\in B_{1}^{+}.
\end{equation*}%
Note that (\ref{harm}) still holds for $w\left( r_{j}y\right) /r_{j}$. Then $%
W_{j}\left( y\right) =0$ on $y_{1}=0$, $W_{j}\left( y\right) \geq -\sigma
_{j}$ on $S_{1}^{+}$, $W_{j}\left( y\right) \geq \delta /2$\ on $B_{c}\left(
y_{j}/r_{j}\right) \cap S_{1}^{+}$. Moreover, setting $\mathcal{L}_{j}'=\mbox{
div}\left( A'\left( r_{j}y\right) \nabla \right) $, 
\begin{equation*}
\mathcal L_{j}^{\prime }W_{j}\leq r_{j}\left\Vert f'\right\Vert _{\infty }\text{ \ \ \
in }B_{1}^{+}\text{. }
\end{equation*}%
By Hopf principle and standard comparison, in $B_{1/2}^{+}$ we can write 
\begin{equation*}
W_{j}\left( y\right) \geq (-c_{3}\sigma _{j}-c_{4}r_{j}\left\Vert
f'\right\Vert _{\infty }+c_{5}\delta /2)y_{1}
\end{equation*}%
with $c_{3,}c_{5}$ universal and $c_{5}$ depending on the Lipschitz constant
of $v$. For $j$ large enough, we get, say%
\begin{equation*}
m_{\infty }w_{j}\left( y\right) -v_{j}\left( y\right) \geq \frac{\delta }{100%
}y_{1}.
\end{equation*}%
Rescaling back and using (\ref{harm}), we get a contradiction to the
definition of $m_{\infty }$. Thus we have equality in (\ref{equality}) and
taking into account (\ref{asym}), we get%
\begin{equation*}
v\left( x\right) =\gamma m_{\infty }y_{1}+o\left( \left\vert y\right\vert
\right) \text{ \ \ as }x\rightarrow 0\text{, }x\in \overline{\Omega }%
^{\prime }\cap B_{1}\text{.}
\end{equation*}%
Going back to the original variables, from (\ref{opiccolo}), we get%
\begin{equation*}
u\left( x\right) =\alpha x_{1}+o\left( \left\vert x\right\vert \right) \ \ 
\text{as }x\rightarrow 0,\text{ }x\in \overline{\Omega }\cap B_{1}
\end{equation*}%
with $\alpha =\left( \tfrac{2}{\rho ^{2}}-1\right) \gamma m_{\infty }.$ \end{proof}

\begin{lem}\label{Lemma_2as} Let  $\Omega $ be an open set, $0\in
\partial \Omega.$ Assume that
\begin{equation}
\mathit{\ }B_{\rho }\left( \rho e^{1}\right) \subset \Omega 
\label{interior}
\end{equation}
 Let $u$ be a nonnegative Lipschitz function in $ 
B_{2}\cap \overline{\Omega },$ satisfying $\mathcal{L}u=f$  in $ 
B_{2}\cap \Omega $ and $u=0$ on $\partial \Omega \cap
B_{2}.$

Then there exists $\alpha \geq 0$ such that 
\begin{equation*}
u\left( x\right) =\alpha x_{1}+o\left( \left\vert x\right\vert \right)\quad \mbox{
 as }\quad x\rightarrow 0,\:\:x\in B_{\rho }\left( \rho e^{1}\right). 
 \end{equation*}
\end{lem}

\begin{proof}
After a smooth change of variables\emph{\ }(e.g. flattening
the surface ball) which leaves both the origin and the normal direction at $0
$ fixed, we may replace (\ref{interior}) by%
\begin{equation*}
B_{2}^{+}\subset \Omega 
\end{equation*}%
always with $0\in \partial \Omega $. We keep the same notation $u$ and $\mathcal{L}$
for the transformed $u$ and the  new operator, which is uniformly elliptic
with ellipticity constant of the same order of $\lambda ,\Lambda $.  As in
Lemma \ref{Lemma_1as}, let $w=w\left( x\right) $ be the $\mathcal{L}-$harmonic measure in $B_{2}^{+}$ of 
$S_{2}^{+}=\partial B_{2}\cap \left\{ y_{1}>0\right\} $. For $k\geq 1,$
integer, define 
\begin{equation*}
\alpha _{k}=\sup \left\{ \alpha :u\left( x\right) \geq \beta w\left(
x\right) \text{ \ \ \ for every }x\in B_{1/k}^{+}\right\} .
\end{equation*}%
Then $\left\{ \alpha _{k}\right\} $ is nondecreasing and $\alpha
_{k}\rightarrow \alpha \geq 0.$ Moreover 
\begin{equation}
u\left( x\right) \geq \alpha w\left( x\right) +o\left( \left\vert
x\right\vert \right) \text{ \ \ as }x\rightarrow 0\text{, }x\in B_{1}^{+}.
\label{asymm}
\end{equation}%
We claim that equality holds in (\ref{equality}). If not, there exist $%
\delta >0$ and a sequence $\left\{ x_{j}\right\} \in B_{1}^{+}$ such that $%
r_{j}=\left\vert x_{j}\right\vert \rightarrow 0$ and 
\begin{equation*}
u\left( x_{j}\right) -\alpha w\left( x_{j}\right) \geq \delta r_{j}.\text{ }
\end{equation*}%
By Lipschitz continuity, we can write%
\begin{equation}
U\left( x\right) \equiv u\left( x\right) -\alpha w\left( x\right) \geq
\delta r_{j}/2\text{ \ \ \ \ \ on }B_{cr_{j}}\left( x_{j}\right) \cap
S_{r_{j}}^{+}  \label{arc}
\end{equation}%
with $c$ depending on $\alpha $ and the Lipschitz constants of $u$ and $w$.

On the other hand, (\ref{equality}) implies that 
\begin{equation}
U\left( x\right) \geq -\sigma _{j}r_{j}\text{ \ \ \ \ \ on }S_{r_{j}}^{+}
\label{arc2}
\end{equation}%
with $\sigma _{j}\rightarrow 0$. Rescale by setting%
\begin{equation*}
U_{j}\left( x\right) =\frac{u\left( r_{j}x\right) }{r_{j}}-\alpha \frac{%
w\left( r_{j}x\right) }{r_{j}}=\frac{U\left( r_{j}x\right) }{r_{j}}\text{ \
\ \ \ \ \ }x\in B_{1}^{+}.
\end{equation*}%
Then $U_{j}\left( 0\right) =0,$ $U_{j}\left( x\right) \geq -\sigma _{j}$ on $%
S_{1}^{+}$, $U_{j}\left( x\right) \geq \delta /2$\ on $B_{c}\left(
x_{j}/r_{j}\right) \cap S_{1}^{+}$. Moreover, setting $\mathcal L_{j}=\mbox{div}%
\left( A\left( r_{j}x\right) \nabla \right) $, 
\begin{equation*}
\mathcal L_{j}U_{j}\leq r_{j}\left\Vert f\right\Vert _{\infty }\text{ \ \ \ in }%
B_{1}^{+}\text{. }
\end{equation*}%
By Hopf principle and standard arguments, in $B_{1/2}^{+}$ we can write, for 
$j$ large 
\begin{equation*}
U_{j}\left( x\right) \geq (-c\sigma _{j}-c_{0}r_{j}\left\Vert f\right\Vert
_{\infty }+C\delta /2)x_{1}\geq \frac{\delta }{100}x_{1}.
\end{equation*}%
Rescaling back and using (\ref{harm}), we get a contradiction to the
definition of $\alpha $. 
\end{proof}

\begin{rem}\label{asy} We remark that the expansions in the lemmas above remain valid if we replace the assumption that $u$ is Lipschitz with the existence of a touching ball at $0$ both from the right and from the left.
\end{rem}

\section{Lipschitz regularity of $u^+.$}\label{lipschitzregu+}

In this section we prove that $u^+$ is Lipschitz continuos. In order to follow the strategy developed in \cite{C3}, we need the following ``almost-monotonicity" formula, see \cite{MP}.

\begin{prop}\label{AM}Let $u_i, i=1,2$ be continuous functions in the unit ball $B_1$ that satisfy
$$
u_i\geq 0, \quad \mathcal L u_{i} \geq -1,\quad u_1\cdot u_2=0\quad \mbox{in}\quad B_1.
$$
Then there exist universal constants $C_0$ and $r_0,$ such that the functional
$$
\Phi(r):= r^{-4} \int_{B_r}\frac{\mid \nabla u_1\mid^2}{\mid x\mid^{n-2}}\int_{B_r} \frac{\mid \nabla u_2\mid^2}{\mid x\mid^{n-2}}dx
$$
satisfies 
$$
\Phi(r)\leq C_0(1+\| u_1\|^2_{L^2(B_1)}+\| u_2\|^2_{L^2(B_1)})^2
$$
for $0<r<r_0.$
\end{prop}

\begin{rem}\label{Fubini}
We remark that, by Fubini's theorem
$$\int_{B_R} \frac{|\nabla u_i|^2}{|x|^{n-2}} dx = R^{2-n} \int_{B_R} |\nabla u_i|^2 dx + (n-2)R^{-2} \int_0^r (\int_{B_r} |\nabla u_i|^2)r^{1-n} dr.$$

\end{rem}

\begin{rem}\label{rmarko6}
We remark that  if $v$ satisfies  $\mathcal{L}v \geq - M$ say in $B_1^+(v)$, then $\mathcal L v^+ \geq -M$ in $B$. This follows by standard arguments. Indeed, if $\psi_\eps(t)$ is a convex increasing function such that $\psi_\eps(t) = 0$ for $t \leq \eps$, then it is easy to see that $$\mathcal L \psi_\eps(v) \geq \psi'(u_\eps) \mathcal Lv \geq - M \quad \text{in $B_1$}.$$ The desired result follows by approximating $t^+$ with a sequence of $\psi_\eps.$
\end{rem}

The next lemma is the first step towards proving that $u^+$ is Lipschitz. The standard technique of harmonic replacement cannot be applied in our case, as we are not imposing any sign condition on the right-hand-side $f_1.$ We bypass this difficulty solving an obstacle-type problem.

\begin{lem} Let $w \in \mathcal F$, then there exists $\tilde w \in \mathcal F$ such that
\begin{itemize}
\item[(i)] $\mathcal L \tilde w = f_1 \quad \text{in $\Omega^+(\tilde w)$}, $
\item[(ii)]  $\tilde w^+ \leq w$, $\tilde w^- = w$
\item[(iii)] $\tilde w \geq \underline u, $ \quad $\tilde w = \phi$ on $\p \Omega.$
\end{itemize}
\end{lem}

\begin{proof} Let $w \in \mathcal F$. For notational simplicity call $\Omega^+ = \Omega^+(w)$ and set 
$$\mathcal S= \{v \in C(\bar \Omega^+)  : \mathcal L v \geq f_1\chi_{\{v>0\}} \ \text{in $\Omega^+$}, v \geq 0 \ \text{in $\Omega^+$}, v=w \  \text{on $\p \Omega^+$}\}.$$ Notice that $\mathcal S \neq \emptyset$ since $\underline{u}^+ \in \mathcal S.$ Also, if $v \in \mathcal S$ then $v \leq w$ in $\Omega^+$. Define,
$$\tilde w : = \sup \mathcal S.$$ Then $\tilde w \leq w$ and solves the obstacle problem (see \cite{KS})
$$\begin{cases}
\mathcal  L {\tilde w} = f_1 \quad \text{in $\{\tilde w>0\}$}, \quad \tilde w \geq 0 \quad \text{in $\Omega^+$}\\
\tilde w = w \quad \text{on $\p \Omega^+$.}
\end{cases}$$
By the regularity theory for the obstacle problem we conclude that $\tilde w$ is locally $C^{1,\gamma}$ in $\Omega^+$ (see \cite{T}).

Extend $\tilde w$ to $\bar \Omega$ by setting
$$\tilde w = w \quad \text{in $\bar \Omega \cap \{w \leq 0\}.$}$$
Hence by definition, $\tilde w \geq \underline u$ on $\Omega$ and $\tilde w =g$ on $\p \Omega$.

To conclude that $\tilde w \in \mathcal F$ we only need to show that $\tilde w$ satisfies the free boundary condition in the sense of Definition \ref{admis}.

Let $x_0 \in F(\tilde w)$, then either $x_0 \in F(w)$ or $x_0 \in \Omega^+ \cap \p\{\tilde w=0\}.$ In the latter case, by the $C^{1,\gamma}$ regularity  of $\tilde w$ we immediately obtain that the free boundary condition in satisfied, possibly with $\alpha=\beta=0$ (recall $G(0)>0.$) If $x_0 \in F(w)$ then the conclusion follows immediately from the fact that $\tilde w \leq w$ in $\Omega^+$ and $\tilde w=w$ otherwise.
\end{proof}

The following result is a consequence of the weak monotonicity formula.
\begin{thm}\label{lipofu+}
Let $w\in \mathcal{F}$ and $\mathcal L w=f_1$ in $\Omega^+(w).$ Then, $w^+$ is locally Lipschitz in $\Omega$. Moreover, 
denoting  by $$G^{-1}(\alpha)=\inf_{x,\nu}G^{-1}(\alpha,x,\nu),$$ 
for any $D \subset \subset \Omega,$ $w^+$ is Lipschitz in $D$ with Lipschitz constant $L_D$ satisfying
\begin{equation}\label{wlip2}L_D G^{-1}(L_D) \leq  C(1+ \|w^+\|^2_{L^2(D)} + \|w^-\|^2_{L^2(D)}) \end{equation} and $C$ depending on $ D.$

\end{thm}
\begin{proof} Let $x_0 \in F(w)$ be a regular point from the left where $w$ has the asymptotic development
$$
w^+= \alpha \langle x-x_0,\nu\rangle^+ +o(\mid x-x_0\mid), \quad \alpha > 0
$$
$$
w^-\geq \beta\langle x-x_0,\nu\rangle^-+o(\mid x-x_0\mid), \quad \beta \geq 0,
$$
with $$\alpha < G(\beta, x_0, \nu).$$
Let us show that
$$\alpha G^{-1}(\alpha) \leq C(1+\| w^+\|^2_{L^2}+\| w^-\|^2_{L^2})^2,$$ with $C$ depending on $dist(x_0,\p \Omega)$. We will use Proposition \ref{AM}. Notice that in view of Remark \ref{rmarko6}, the conclusion of Proposition \ref{AM} holds for $u_1=w^+, u_2=w^-.$

If $G^{-1}(\alpha)=0,$ then there is nothing to prove. Thus, let 
$G^{-1}(\alpha)>0$ and let us prove that 
$$
\alpha^2\beta^2\leq  C(1+\| w^+\|^2_{L^2}+\| w^-\|^2_{L^2})^2,
$$
from which the desired inequality will follow.

For convenience, use coordinates $x=(x',y) \in \R^{n-1} \times \R$ and assume that $x_0=(0,0), \nu=(0,1).$ 
Following \cite{C3} pg. 587, one can estimate that as $s \to 0,$
$$\int_{B_s} |\nabla w^+|^2 dx \geq \int_{B_s \cap \{y>0\}} (\alpha^2 - o(1)) dx$$
and 
$$\int_{B_s} |\nabla w^-|^2 dx \geq \int_{B_s \cap \{y>0\}} (\beta^2 - o(1)) dx.$$

Thus, by Remark \ref{Fubini}, for all sufficiently small $s$, if $\Phi$ is the functional defined in Proposition \ref{AM}

\begin{equation}\label{phi0}\Phi(s) \geq c_n s^{-4} \int_0^s (\alpha^2 -o(1))r dr \int_0^s (\beta^2-o(1)) rdr,\end{equation}
with $c_n=16/\omega_n^2$ ($\omega_n$ the measure of the unit sphere.)

Hence, for all $r$ small
$\alpha^2 \beta^2 \leq C\Phi(r)$
which together with the conclusion of  Proposition \ref{AM} gives the desired estimate.

Now, let $x_0 \in \Omega^+(w) \cap D$ and let $$dist(x_0, F(w)) = |x_0-y_0| = r < \frac 1 2  dist(D, \p \Omega),$$ say $r \leq 1.$

To prove the result it is sufficient to prove the existence of a positive constant  $M$ such that
\begin{equation*}
\frac{w\left( x_{0}\right) }{r}\leq M.
\end{equation*}
Suppose $$\frac{w\left( x_{0}\right) }{r} > M,$$
with $M$ to be specified later.
By Lemma \ref{barrier}, we have that
$$w(x) \geq \alpha_M \langle x -y_0, \nu \rangle^+ + o(|x-y_0|)$$
with
$$\alpha_M= \bar c M - \bar C r\|f_1\|_\infty.
$$
For $M$ large $\alpha_M > 0$ and $y_0$ is regular from the left. Then according to Lemma \ref{asy} 
$$w(x) = \alpha \langle x -y_0, \nu \rangle^+ + o(|x-y_0|)$$ and $\alpha \geq \alpha_M.$

 Hence we can apply the previous estimate and conclude that
$$\alpha_M G^{-1}(\alpha_M) \leq C(1+\| w^+\|^2_{L^2}+\| w^-\|^2_{L^2})^2.$$
The contradiction follows for $M$ large, since $\alpha_M G^{-1}(\alpha_M) \to \infty$ as $M \to \infty.$
\end{proof}

As a corollary of the two results above we obtain the following.

\begin{cor} $u^+$ is locally Lispchitz and it satisfies 
$$\mathcal L u = f_1 \quad \text{in $\Omega^+(u)$}.$$
\end{cor}

\section{The function $u$ is Lipschitz}\label{uLipschitz}

In this section we show that $u^-$ is Lipschitz. First, we prove the following standard lemma.

\begin{lem}
If $w_1, w_2\in \mathcal{F},$ then $$w^*=\min\{w_1,w_2\}.$$ 
\end{lem}
\begin{proof}
The fact that $\mathcal L w^*= f_1$ in $\Omega^+(w^*)$ and $\mathcal L w^* = f_2 \chi_{\{w^* <0\}}$ follows from standard arguments. To prove that $w^*$ solves the free boundary condition in the sense of Definition \ref{admis} it is suffices to notice that $\Omega^+(w^*) = \Omega^+(w_1) \cap \Omega^+(w_2).$ Hence, any ball touching $F(w_1)$ or $F(w_2)$ from the left will also touch $F(w^*)$ from the left. Thus, if $x_0$ is not regular for $F(w^*)$, it cannot be regular for $F(w_1)$ and $F(w_2)$ either and near $x_0$, $w^* (x)=w_1(x)=w_2(x)=o(|x-x_0|)$. If $x_0$ is regular, the the asymptotic developments for $w^*$ come from those of $w_1$ or $w_2$ (possibly with $\alpha=\beta=0$ in case $x_0$ is not regular for either $F(w_1)$ or $F(w_2)$.)
\end{proof}

The proof of the Lipschitz continuity of $u^-$ is based on the following replacement technique. The harmonic replacement technique in \cite{C3} does not work in this context. Instead, we need to perform a replacement with solutions to obstacle-type problems.

Precisely, let $w \in \mathcal F$ and let $w(x_0) < 0.$ Let $B:=B_R(x_0)$ be a ball around $x_0$.
Denote by $$\Omega_1: = \Omega^+(w) \setminus \bar B.$$Define $$\mathcal{S}_1= \{v : \mathcal L v \geq f_1\chi_{\{v>0\}} \ \text{in $\Omega_1$}, v \geq 0 \ \text{in $\Omega_1$}, v=w \ \text{on $\p \Omega_1 \setminus \p B$}, v=0 \ \text{on $\p B$}\}.$$

Notice that since $\underline u$ is locally Lipschitz and $\underline u(x_0) \leq  w(x_0) < 0$, $\underline u$ is strictly negative in $B$ for $R$ small. Hence $\underline{u}^+ \in \mathcal{S}_1$ and $\mathcal{S}_1$ is non empty.
Let $$w_1= \sup \mathcal{S}_1.$$
Then $w_1$ solves the obstacle problem (see \cite{KS, T})
\begin{equation}\label{w1}
\mathcal L w_1 = f_1 \quad \text{in $\{w_1>0\}$}, \quad w_1 \geq 0.
\end{equation}

Analogously, define
$$\mathcal{S}_2 = \{v : \mathcal L v \geq -f_2\chi_{\{v>0\}} \ \text{in $B$}, v \geq 0 \ \text{in $B$}, \ v = w^- \ \text{on $\p B$}\}.$$
Clearly $\mathcal{S}_2 \neq \emptyset$ since $w^- \in \mathcal{S}_2$. 
 Let,
$$w_2 = \sup S_2.$$
Again, $w_2$ solves the obstacle problem
\begin{equation}\label{w2}\mathcal L w_2 = -f_2 \quad \text{in $\{w_2>0\}$}, \quad w_2 \geq 0.\end{equation}

We define the ``double-replacement" $\tilde w$ of $w$ in $B$ as follows
$$\tilde w = \begin{cases}
w_1 \quad \text{in $\bar \Omega_1$}\\
-w_2 \quad \text{in $\bar B$}\\
w \quad \text{otherwise.}
\end{cases}$$
By construction $\tilde w \leq w.$ Indeed in $\Omega_1$ this follows by the maximum principle, while in $B$ it follows from the fact that $w^- \in \mathcal S_2.$

We wish to prove the following lemma.

\begin{lem}\label{lipmain} Let $w \in \mathcal F, w(x_0) = -h.$ Then there exists $\epsilon$ (depending on $dist(x_0,\p \Omega)$ and $\underline{u}$) such that
\begin{itemize}
\item[(i)] The double replacement $\tilde w$ of $w$ in $B_{\epsilon h}(x_0)$ belongs to $\mathcal F$ and $\underline u \leq \tilde w \leq w.$
\item[(ii)] $\mathcal L \tilde w = f_2$ and $\tilde w <0$ in $B_{\eps h}(x_0)$ and 
$$|\nabla \tilde w| \leq \frac{C}{\epsilon} + \epsilon C \|f_2\|_\infty \quad \text{in $B_{\epsilon h/2}(x_0).$}$$ 
\end{itemize}
\end{lem}
\begin{proof} We already noticed that $\tilde w \leq w$. No we observe that $\tilde w \geq \underline u.$ As already remarked, since $\underline{u}$ is locally Lipschitz and $\underline{u}(x_0) < -h$, $\underline{u}$ is strictly negative in $B:=B_{\epsilon h} (x_0)$ for $\epsilon$ small. Thus, $\underline{u}^+ \in \mathcal S_1$ and $w_1 \geq \underline{u}.$ Also, by the maximum principle in $\{w_2 >0\}$, it follows that $-w_2 \geq \underline{u}$ in $B$. Hence $\tilde w \geq \underline u.$ We denote with $-m = \min_{\overline{\Omega}} \underline u$, $m >0.$ Notice that $h \leq m.$

Let us also observe that, for $\epsilon$ small, $w_2 >0$ in $B_{\epsilon h}(x_0).$ Indeed if $\p \{w_2 >0\} \cap B_{\epsilon h}(x_0) \neq \emptyset$, then by the growth of the solution to the obstacle problem, we get 
$$w_2(x_0) \leq C (\eps h)^{2}.$$ For $\epsilon$ small, this contradicts that $ - w_2(x_0) \leq w(x_0) = -h.$
In particular, it follows from \eqref{w2} that
\begin{equation}\label{star}\mathcal L \tilde w = f_2 \quad \text{in $B_{\eps h}(x_0).$}\end{equation}
Now, the fact that $\tilde w$ satisfies (a) in Definition \ref{admis} follows from \eqref{w1}-\eqref{star} and standard arguments.

We need to verify that the free boundary condition is satisfied in the sense of part (b) in Definition \ref{admis}. Let $\bar x \in F(\tilde w)$. Then three possibilities can occur. If $x_1 \in F(w)$, we use that $\tilde w \leq w$ and hence $\tilde w$ has the correct asymptotic behavior whether $x_1$ is regular or not (recall $G(0, \cdot, \cdot)>0$). 
If $x_1 \in \p \{w_1>0\} \cap \Omega^+(w)$ then by the regularity of the solution to the obstacle problem we get again that $\tilde w$ has the correct asymptotic behavior. Finally, we consider the case when $x_1 \in \p B \cap \Omega^+(w).$ Since $w^+$ is locally Lipschitz, say with constant $L$ in $B_{d_0/2}(x_0)$ we get that
$$\tilde w \leq w^+ \leq L \eps h \quad \text{in $B_{2 \eps h}(x_0)$}.$$
Let us rescale and using the notation in \eqref{notation} call
$$\tilde w_\epsilon(x) = \frac{\tilde w(x_0 + \eps h x)}{\eps h}.$$
Then,
$$\tilde w_\eps \leq L \quad \text{in $B_2.$}$$

Let us call $v_1,v_2$ the solutions to the  the following problems:
$$\mathcal L_\eps v_1 =0 \quad \text{in $B_{2} \setminus \overline{B_{1}}$}$$
$$v_1=0 \quad \text{on $\p B_{1}$}, \quad v_1=1 \quad \text{on $\p B_{2},$}$$ 
$$\mathcal L_\eps v_2 =-1 \quad \text{in $B_{2} \setminus \overline{B_{1}}$}$$
$$v_2=0 \quad \text{on $\p B_{1}$}, \quad v_2=0 \quad \text{on $\p B_{2}.$}$$ 


Define,
$$v := L v_1 + \eps h M v_2$$
with $M= \|f_1\|_\infty.$
Then, applying the maximum principle in $(B_2 \setminus \overline B_1) \cap \Omega^+(\tilde w _\eps)$ we obtain that
$$\tilde w_\eps^+ \leq v \quad \text{in $B_2 \setminus \overline B_1.$}$$
Thus,
$$\tilde w_\eps^+ \leq \alpha \langle x - (x_1)_\eps, \nu((x_1)_\eps) \rangle^+ + o(|x-(x_1)_\eps|)$$
with 
$$\alpha = (L c_1 + \eps h M c_2), \quad c_1 = (v_1)_\nu|_{\p B_1}, \quad c_2=(v_2)_\nu|_{\p B_1}$$

$$(x_1)_\eps = \frac{x_1-x_0}{\eps h}$$ and $\nu(y)$ the normal to $\p B_1$ at $y$ pointing outside $B_1.$

In terms of $\tilde w$ this gives,
\begin{equation}\label{alpha} \tilde w^+ \leq \alpha \langle x - x_1, \nu \rangle^+ + o(|x-x_1|), \quad \alpha \leq \bar L,\end{equation}
and $\nu$ the exterior unit normal to $\p B_{\eps h}(x_0)$ at $x_1$.

On the other hand, by Lemma \ref{barrier} applied to $-(w_2)_\eps$ we have that
$$\tilde w_\eps^- = -(w_2)_\eps \geq \beta \langle x - (x_1)_\eps, \nu((x_1)_\eps) \rangle^- + o(|x-(x_1)_\eps|), $$
with ($\bar M = \|f_2\|_\infty$)
$$\beta \geq \frac{\bar c}{\eps} - \bar C \eps h \bar M.$$

In terms of $\tilde w$ and for $\eps$ small, this implies that
\begin{equation}\label{beta}\tilde w^-  \geq \beta \langle x - x_1, \nu \rangle^- + o(|x-x_1|), \quad \beta \geq \frac{\bar c}{2\eps}.\end{equation}
In view of \eqref{alpha}-\eqref{beta}, the free boundary condition is satisfied if we choose $\eps$ small enough so that
$$\bar L < inf_{x, \nu}G(\frac{\bar c}{\eps}, \cdot, \cdot).$$

Finally, the estimate in (ii) follows from standard Schauder estimates and Harnack inequality. 
\end{proof}

We obtain the following immediate corollary.

\begin{cor}\label{intermres}
Let $x_0$ be a point where  $u(x_0)=-h<0. $ Then, there exists a non-increasing sequence $\{\tilde w_j\} \subset \mathcal{F}, \tilde w \geq \underline u,$ and $\eps >0,$ depending on $d_0=dist(x_0, \p \Omega)$, such that the following hold: 
\begin{itemize}
\item[(i)] $\tilde w_k(x_0) \searrow u(x_0);$
\item[(ii)] $\mathcal L \tilde w_k = f_2$ and $ \tilde w_k <0$ in $B_{\eps h}(x_0);$
\item[(iii)] For each $k$, $\tilde w_k$ is Lipschitz in $B_{\eps h/2}(x_0)$ with Lipschitz constant $L_0$ depending on $d_0.$
\end{itemize} 
\end{cor}

Finally, we can finish the proof that $u$ satisfies part a) in Definition \ref{visc}.

\begin{cor}
$u$ is locally Lipschitz in $\Omega,$ continuous in $\overline \Omega$, $u=g$ on $\p \Omega$. Moreover $u$ solves 
$$ \mathcal L  u = f_2 \chi_{\{u<0\}}, \quad\text{in  $\Omega^-(u)$}.$$ 
\end{cor}
\begin{proof} Let $u(x_0)=-h <0$ and let $\{\tilde w _k\}$ be as in the lemma above. We want  to prove that $\tilde w_k \searrow u$ uniformly, say on 
$B_{h\epsilon/4}.$ Indeed suppose by contradiction that there exists $x_1\in B_{\epsilon h/4}(x_0)$ where $\tilde{w}(x_1)=\lim_{j\to \infty}\tilde{w}_j(x_1)>u(x_1).$ Then consider a new sequence $\{v_j\}_{j\in \mathbb{N}}$ converging to $u$ at $x_1,$ and define $\{\tilde{u}_k\}$ as a replacement of  $\{\min\{\tilde{v}_k,\tilde{w}_k\}\}_{k\in \mathbb{N}}$ in $B_{\epsilon h/2}(x_0).$ Then $\lim_{k\to\infty}\tilde{u}_k=\tilde{u}$ decreasing with $\tilde{u}\leq\tilde{w}$ in $B_{\epsilon h /2}(x_0),$ $\tilde{u}(x_0)=\tilde{w}(x_0)$ and $\tilde{u}(x_1)<\tilde{w}(x_1).$ Moreover in $B_{\epsilon h/4}(x_0)$
$$
\mathcal L ( \tilde{w}-\tilde{u})=0,
$$
$\tilde{w}-\tilde{u}\geq 0$ and $(\tilde{w}-\tilde{u})(x_0)=0$ and by maximum principle it follows that $\tilde{w}-\tilde{u}\equiv 0$ obtaining a contradiction with $(\tilde{w}-\tilde{u})(x_1)>0.$ As a consequence $u$ satisfies
$$
\mathcal L u=f_2 \quad \text{in $\{u<0\}$.}
$$
\end{proof}

\begin{cor}\label{unegsuicompact}
 If $K$ is compactly contained in $\Omega,$  then $u$ is uniform limit of a 
sequence of functions $\{w_k\}_{k\in \mathbb{N}}\subset \mathcal{F}$ in $K.$ If
$K\Subset \Omega^-(u),$ 
$\{w_k\}_{k\in \mathbb{N}},$ may be taken non-positive in $\overline{K}.$
\end{cor}
\begin{proof}
The first part follows from the fact that  $\{w^+: \ w\in \mathcal{F}\}$ is  equilipschitz  in $\overline{K}$ 
and from  the previous  replacement technique.

By compactness, it is enough to prove the second part for balls $B_{\eps}(x_0) \Subset \Omega^-(u)$, with $\eps$ small enough. 
Let $w_k \searrow u$ uniformly in $\overline B_{2\eps}(x_0) \Subset \Omega^-(u).$
Let us rescale by $\eps$ and use the notation in \eqref{notation}. 
Let $\eta, \xi$ solve the following problems:
$$\mathcal L_\eps \eta =0 \quad \text{in $B_{2} \setminus \overline{B_{1}}$}$$
$$\eta=0 \quad \text{on $\p B_{1}$}, \quad \eta=1 \quad \text{on $\p B_{2},$}$$ 
$$\mathcal L_\eps \xi =-1 \quad \text{in $B_{2} \setminus \overline{B_{1}}$}$$
$$\xi=0 \quad \text{on $\p B_{1}$}, \quad \xi=0 \quad \text{on $\p B_{2}.$}$$ 

Call 
$$c_0:= \eta_{\nu}|_{\p B_{1}}>0, \quad c_1:= \xi_{\nu}|_{\p B_{1}}>0$$
with $\nu$ the unit normal to $\p B_{1}$ pointing inward and let $c_2>0$ be such that \begin{equation}\label{c2*}c_0 c_2 < \frac{G(0)}{2}.\end{equation}

Define, ($M= \|f_2\|_\infty$, say $M>0$)
$$v := \eps M \xi + c_2\eta \quad \text{in $B_{2} \setminus \overline{B_{1}}$, } \quad v\equiv 0 \quad \text{on $B_1$.}$$

Then,

\begin{equation*} \mathcal L_\eps v= -\eps M  \leq \eps f_2^\eps\quad \text{in $B_{2} \setminus \overline{B_{1}}$.}\end{equation*}

Since $u_\eps \leq 0$ in $\overline B_2$, for $k$  sufficiently large  $w_k\leq \epsilon M/2$ in $\overline B_2$. Define
$$
\bar{w}_k=\left\{\begin{array}{l}
\min\{w_k^\eps,v\}\quad,\mbox{in}\quad\overline B_2,\\
w^\eps_k,\quad\mbox{ otherwise}.
\end{array}
\right.
$$
Then, in view of \eqref{c2*}, as long as $$\eps < \frac{G(0)}{2M c_1}$$ the function
$$\bar w_k(x) = \eps \bar w_k^\eps(\frac{x-x_0}{\eps})$$ satisfies $$ \bar w_k \in \mathcal F, \quad \bar w_k \leq 0 \quad \text{in $\overline B_\eps(x_0)$}$$
and $\bar w_k \searrow u$ in $\overline B_{\eps}(x_0),$ as desired.

\end{proof}
\section{On the degeneracy of $u^+$}\label{ondegeneracyu+}
In this section we prove that $u^+$ is not degenerate.
As a consequence $F(w_k)\to F(u)$ locally in Hausdorff distance and $\chi_{\{w_k>0\}}\to \chi_{\{u>0\}}$ in $L^1_{loc}(\Omega).$

First, we recall the following standard lemma.

\begin{lem}
Let $u$ be a Lipschitz  function in $\overline{\Omega}\cap B_1(0)$ satisfying $\mathcal L u=f,$ vanishing on $\partial\Omega\cap B_1$ and $0\in\partial \Omega.$ Suppose that  there exists a positive constant $C$ such that for every $x\in B_{1/2} \cap \Omega$
\begin{equation}\label{nondeg}
u(x)\geq c\mbox{dist}(x,\partial \Omega).
\end{equation}
Then there exists a constant $C>0$ such that 
$$
\sup_{B_r(0)}u\geq Cr,
$$
for all $r \leq r_0$ universal.
\end{lem}
\begin{proof}
Let $\mbox{dist}(x_0,\partial \Omega)=\epsilon.$ Then by (\ref{nondeg}) and the  Lipschitz continuity of $u$ (say $L=Lip(u)$)
$$c\epsilon\leq  u(x_0)\leq L \epsilon.
$$

We wish to show that there exists $x_1 \in B_\eps(x_0)$ such that $$u(x_1)\geq (1+\delta)u(x_0),$$
with $\delta$ to be specified later.

Assume not, then
$$v(x) : = (1+\delta)u(x_0) - u(x) > 0 \quad \text{in $B_{\eps}(x_0)$}$$
and solves 
$$\mathcal  L v = -f  \quad \text{in $B_{\eps}(x_0)$}.$$

By Harnack inequality, 

$$v \leq C(L)(\delta u(x_0) + \eps^2 \|f\|_\infty) \quad \text{in $\overline B_{c(L)\eps}(x_0)$},$$
with $c(L)=1 - \frac{c}{4L}.$

Hence, for $\delta < c/4L, \eps < c/4\|f\|_\infty,$

$$v \leq C(L)(\delta L \eps + \eps^2 \|f\|_\infty) \leq \frac{1}{2}c \eps \leq \frac{u(x_0)}{2} \quad \text{in $\overline B_{c(L)\eps}(x_0)$}.$$

From the definition of $v$ it follows that

$$u \geq \frac{c \eps}{2} \quad \text{in $B_{c(L)\eps}(x_0)$}.$$

However, from the Lipschitz continuity of $u$ it follows that
$$u(x) \leq L (1-c(L))\eps = c \frac{\eps}{4}  \quad \text{on $ \p B_{c(L)\eps}(x_0)$}$$
a contradiction.

Thus we can construct inductively a sequence of points $x_k$ such that $$u(x_{k+1})=(1+\delta)u(x_k), \quad |x_{k+1}- x_k| \leq C d(x_k, \p \Omega).$$

Then using the fact that $dist(x_k, \p \Omega) \sim u(x_k)$ and that $u(x_k)$ grows geometrically we find \begin{align*}|x_{k+1} - x_0| &\leq \sum_{i=0}^{k} |x_{i+1}-x_i|  \leq C \sum_{i=0}^k dist(x_{i}, \p \Omega) \\ & \leq C \sum_{i=0}^k u(x_{i}) \leq C u(x_{k+1}) \sim dist(x_{k+1}, \p \Omega).\end{align*} 
Hence for a sequence of $r_k$'s of size $u(x_k)$ we have that $$\sup_{B_{r_k}(x_0)} u \geq c r_k$$ from which we obtain that 
 $$\sup_{B_{r}(x_0)} u \geq c r, \quad \text{for all $r \geq |x_0|.$}$$ 
The conclusion follows by letting $x_0$ go to 0.

\end{proof}

\begin{lem} There exist universal constants $\bar r, \bar C >0$, such that 
$$u(x_0)\geq \bar C dist(x_0, F(u)), \quad \text{in $\{x \in \Omega^+(u) : dist(x, F(u)) \leq \bar r\}$}.$$
\end{lem}
\begin{proof}
Let $x_0 \in \Omega^+(u)$, $r=dist(x_0, F(u))$ with $r \leq \bar r$ universal to be specified later. Assume first that $$dist(x_0, \Omega^+(\underline{u})) > \frac r 2.$$ Thus, 
\begin{equation}\label{neg}
\underline{u} \leq 0 \quad \text{in $B_{r/2}(x_0).$}\end{equation} Let $w_k \in \mathcal F$ converge uniformly to $u$, say in $B_R(x_0), r \leq R.$ Let us rescale by $r$ around $x_0$ and use the notation \eqref{notation}. Then $u_r$ solves the free boundary problem
\begin{equation}\label{FB}\begin{cases}\mathcal L_r u_r = r f_1^r \quad \text{in $\Omega_r^+(u_r)$}\\
\mathcal L_r u_r= r f_2^r\chi_{\{u_r <0\}}\quad \text{in $\Omega_r^-(u_r)$}\\
(u_r)_{\nu}^+=G_r((u_r)^-_{\nu}, x, \nu) \quad \text{on $F(u_r).$}
\end{cases}
\end{equation}

Moreover $w_k^r$ converges to $u_r$ uniformly in $B_{R/r}.$ Clearly, $u_r$ is the infimum of all admissible supersolutions to \eqref{FB} (in the sense of Definition \ref{admis}) which are above the rescaling $\underline{u}_r(x) = \frac{\underline{u}(x_0 +rx)}{r}.$

We wish to prove that $$u_r(0) \geq \bar C,$$ with $\bar C>0$ universal, to be specified later. Assume by contradiction $$u_r(0) < \bar C.$$ 

By Harnack inequality in $B_1 \subset \Omega_r^+(u_r)$, we have that
$$u_r \leq C(\bar C + rM), \quad \text{in $B_{1/2}$}$$ where $\|f_1\|_{L^\infty} = M.$ Hence, for $k$ large enough
$$0 < w_k^r \leq C(\bar C +rM) \quad \text{in $B_{1/2}.$}$$

Now, as in Corollary \ref{unegsuicompact}, let $\eta, \xi$ solve the following problems:
$$\mathcal L_r \eta =0 \quad \text{in $B_{1/2} \setminus \overline{B_{1/4}}$}$$
$$\eta=0 \quad \text{on $\p B_{1/4}$}, \quad \eta=1 \quad \text{on $\p B_{1/2},$}$$ 
$$\mathcal L_r  \xi =-1 \quad \text{in $B_{1/2} \setminus \overline{B_{1/4}}$}$$
$$\xi=0 \quad \text{on $\p B_{1/4}$}, \quad \xi=0 \quad \text{on $\p B_{1/2}.$}$$ 

Call 
$$c_0:= \eta_{\nu}|_{\p B_{1/4}}>0, \quad c_1:= \xi_{\nu}|_{\p B_{1/4}}>0$$
with $\nu$ the unit normal to $\p B_{1/4}$ pointing inward and let $c_2>0$ be such that \begin{equation}\label{c2}c_0 c_2 < \frac{G(0)}{2}.\end{equation}

Define,
$$v := rM \xi + c_2\eta \quad \text{in $B_{1/2} \setminus \overline{B_{1/4}}$.}$$

Then,

\begin{equation}\label{super}\mathcal L_r  v= -rM \quad \text{in $B_{1/2} \setminus \overline{B_{1/4}}$.}\end{equation}

Moreover, (say $M>0$) if
\begin{equation}\label{r1}r < \frac{c_2}{2CM}=r_1\end{equation}
and $\bar C$ is chosen so that
\begin{equation}\label{sigma} \bar C \leq \frac{c_2}{2C}\end{equation}
then,

\begin{equation}\label{cont}0< w_k^r \leq \frac{c_2}{2} \leq v \quad \text{on $\p B_{1/2}$.}\end{equation}

Now, define

$$\bar w_r = \begin{cases} w_k^r \quad \text{in $\Omega_r \setminus B_{1/2}$}\\
\min\{w_k^r, v\} \quad \text{in $B_{1/2} \setminus \p B_{1/4}$}\\
0 \quad \text{in $B_{1/4}$.}
\end{cases}$$

This function is continuous in view of \eqref{cont}. Also, from \eqref{super} and the fact that $w_k^r >0$ in $B_{1/2}$ it follows that
$$\begin{cases}
\mathcal L_r \bar w_r \leq rf_1^r \quad \text{in $\Omega_r^+(\bar w_r)$,}\\
\mathcal L_r  \bar w_r \leq rf_2^r\chi_{\{\bar w_r <0\}} \quad \text{in $\Omega^-_r(\bar w_r),$}
\end{cases}$$
and from \eqref{neg}
$$\bar w_r \geq \underline{u}_r \quad \text{in $\Omega_r.$}$$

Thus, to assure that $\bar w_r$ is an admissible supersolution, we need to require that
$$rM c_1 + c_2c_0 < G(0).$$

In view of \eqref{c2}, it is enough to choose
$$r \leq \frac{G(0)}{2 M c_1}=r_2.$$ Thus, for $r \leq \bar r := \min\{r_1,r_2\}$ we have reached a contradiction to the minimality of $u_r$ since $$\bar w_r(0) = 0 < u_r(0).$$

\end{proof}

As a consequence of the two lemmas above, we obtain the following corollary.

\begin{lem}
Let $x\in F(u)$ and let $A$ be a connected component of $\Omega^+(u)\cap (B_{r}(x)\setminus \overline{B}_{r/2}(x))$
such that
$$
\overline{A}\cap \partial B_{r/2}(x)\not=\emptyset,\quad\overline{A}\cap\partial B_{r}(x)\not=\emptyset,
$$
for $r \leq r_0$ universal.
Then 
$$
\sup_{A}u\geq Cr.
$$
Moreover 
$$
\frac{|A\cap B_{r}(x)|}{|B_r(x)|}\geq C>0,
$$
where all the constants $C$ depend on $d(x,\partial \Omega)$ and on $\underline{u}.$
\end{lem}

\section{The function $u$ is a supersolution}\label{supersolutionuf}

 In this section we prove that $u$ satisfies part (a) in Definition \ref{visc}. First we need to the following preliminary result.

\begin{lem}\label{sandro} Let $v_k \geq 0$ satisfy $$\mathcal L v_k  \in L^\infty \quad \text{in $B_2 \cap \{v_k>0\}$}.$$ Assume that $v_k \to v$ uniformly in $B_2$. Then
$$\int_{B_1} \frac{|\nabla v_k|^2}{|x|^{n-2}} dx \to \int_{B_1} \frac {|\nabla v|^2}{|x|^{n-2}} dx.$$
\end{lem}
\begin{proof}
We sketch the proof. Let $V$ be the fundamental solution of the operator $\mathcal{L}.$ Then $V\sim\left\vert x\right\vert ^{2-n}$ (see  \cite{LSW}).

Take a cut-off $\eta \in C_{0}^{\infty }\left( B_{2}\right) $, $\eta =1$ in $%
B_{1}$.  For $w=v$ or $w=v_k$ we have:
\begin{equation}\label{dismain1}
 A\left( x\right) \nabla v\cdot \nabla w=\frac{1}{2} \mathcal{L}\left( w^{2}\right) - w%
\mathcal L w.
\end{equation}
On the other hand, 

\begin{eqnarray*}
\int_{B_{2}}\eta ^{2}V\mathcal{L}\left( w^{2}\right)dx  &=&-\int_{B_{2}}A\left( x\right)
\nabla (\eta ^{2}V)\cdot \nabla \left( w^{2}\right)dx \\
&=&-2\int_{B_{2}}w\nabla w\cdot (A\left( x\right) [2\eta V\nabla \eta +\eta
^{2}\nabla V])dx\\
&=&-4\int_{B_{2}\backslash B_{1}}w\eta V\nabla w\cdot A\left( x\right)
\nabla \eta dx -\int_{B_{2}}A\left( x\right) \nabla V\cdot \nabla \left(
w^{2}\eta ^{2}\right)dx\\
& +&\int_{B_{2}\backslash B_{1}}w^{2}A\left( x\right)
\nabla V\cdot \nabla \left( \eta ^{2}\right)dx \\
&=&-4\int_{B_{2}\backslash B_{1}}w\eta V\nabla w\cdot A\left( x\right)
\nabla \eta dx - w^2(0)\\
&+& \int_{B_{2}\backslash B_{1}}w^{2}A\left( x\right)
\nabla V\cdot \nabla \left( \eta ^{2}\right)dx .
\end{eqnarray*}%
Thus we deduce from \eqref{dismain1} that
$$\int_{B_2} \eta^2 V A(x) \nabla(v_k - v) \cdot \nabla(v_k -v) dx \to 0 \quad \text{as $k \to \infty$.}$$
From this the desired result follows using ellipticity and the estimate on $V.$
\end{proof}

We also need the following variant of the monotonicity formula in Proposition \ref{AM} (see again \cite{MP}). We use the same notation as in Proposition \ref{AM}.

\begin{prop}\label{variant} Assume that 
$$u_i (x) \leq \sigma(|x|), \quad x \in B_1, \quad i=1,2$$ for a Dini modulus of continuity $\sigma(r).$ Then
$$\Phi(\rho) \leq [1+\omega(r)]\phi(r) + C \omega(r), \quad 0<\rho \leq r \leq r_0,$$
with $$\omega(r) \to 0 \quad \text{as $r \to 0^+$}$$ and $C$ depending on $\|u_i\|_{L^2(B_1)}, \sigma, [A]_{0,\gamma}.$ \end{prop}

In view of the expansion Lemmas \ref{Lemma_1as}, \ref{Lemma_2as}, we only need to prove the next result.

\begin{lem}
Let $x_0 \in F(u)$ and
$$
u^+(x)=\alpha\langle x-x_0,\nu\rangle^++o(\vert x-x_0\vert),
$$
and
$$
u^-=\beta\langle x-x_0,\nu\rangle^-+o(\vert x-x_0\vert).
$$
Then,
 $$
 \alpha\leq G(\beta,x_0, \nu).
 $$
\end{lem}
\begin{proof}
Let $\{w_k\}_{k\in\mathbb{N}}\subset\mathcal{F}$ be a uniformly decreasing to $u.$   As a consequence  $w_k$ cannot remain strictly positive  in a neighborhood of $x_0,$ say in a ball $B_r(x_0)$, for all $k$ large. Otherwise $u$ would be a non-negative solution of $\mathcal Lu =f_1$ in such neighborhood. Then, by standard regularity theory $u\in C^{1,\gamma}$ and $\nabla u(x_0)=0.$ Hence, $u^+_\nu(x_0)=0$ contradicting the non-degeneracy of $u^+.$

For each $w_k,$ let 
$$
B_{m,k}=B_{\lambda_{m,k}}(x_0+\frac{1}{m}\nu)
$$
be the largest ball centered at $x_0+\frac{1}{m}\nu$ contained in $\Omega^+(w_k),$ touching $F(w_k)$ at $x_{m,k}$ where $\nu_{m,k}$ is the unit inward normal of $F(w_k)$ at $x_{m,k}.$ Then up to proper subsequences we deduce that
$$
\lambda_{m,k}\to\lambda_m,\quad x_{m,k}\to x_m,\quad\nu_{m,k}\to\nu_m
$$
and
 $B_{\lambda_m}(x_0+\frac{1}{m}\nu)$ touching $F(u)$ at $x_m,$ with unit inward normal $\nu_m.$  From the behavior of $u^+$, we get that
 $$
 \vert x_m-x_0\vert=o(\frac{1}{m}),
 $$
 $$
 \frac{1}{m}+o(\frac{1}{m})\leq \lambda_m\leq \frac{1}{m}
 $$
 and
 $$
 \vert\nu_m-\nu\vert=o(1).
 $$

 Now since
 $w_k\in \mathcal{F}$, near $x_{m,k}$ in $B_{m,k}:$
 $$
 w_k^+\leq \alpha_{m,k}\langle x-x_{m,k},\nu_{m,k}\rangle^++o(\vert x-x_{m,k}\vert)
 $$
 and in 
 $\Omega\setminus B_{m,k}$
 $$
 w_k^- \geq {\beta_{m,k}}\langle x-x_{m,k},\nu_{m,k}\rangle^{-}+o(\vert x-x_{m,k}\vert)
 $$
 with 
 $$
 0 \leq \alpha_{m,k}\leq G(\beta_{m,k}, x_{m,k}, \nu_{m,k}),
 $$ (by Lemma \ref{barrier} the touching occurs at a regular point, for $m,k$ large.)
 We know that
 $$
 w_k^+\geq u^+\geq \alpha\langle x-x_0,\nu\rangle^++o(\vert x-x_0\vert),
 $$
 hence $$\underline{\alpha}_m=\liminf_{k\to \infty}\alpha_{m,k}\geq \alpha-\epsilon_m$$
 and $\epsilon_m\to 0,$ as $m\to \infty.$
 We have to prove that
 $$
 \underline{\beta}=\liminf_{m,k\to +\infty}\beta_{m,k}\leq \beta.
 $$
 To do this we argue as follows. If $\underline{\beta}=0$ there is nothing to prove. Hence let $\beta_{m,k}>0.$ Given $r, \bar x, v$ denote by $$
\Phi_r(\bar x, v)= r^{-4} \int_{B_r(\bar x)}\frac{\mid \nabla v^+\mid^2}{\mid x-\bar x\mid^{n-2}}dx\int_{B_r(\bar x)} \frac{\mid \nabla v^-\mid^2}{\mid x-\bar x\mid^{n-2}}dx.
$$

From \eqref{phi0} in Theorem \ref{lipofu+} we obtain that ($\rho$ small)
\begin{equation*}
 \Phi_\rho(x_{m,k}, w_k) \geq c_n \; \alpha_{m,k}^2\beta_{m,k}^2 + o(1),
\end{equation*}
with $o(1) \to 0$ as $\rho \to 0.$

Thus, using Proposition \ref{variant} and letting $\rho \to 0$, we get  that ($r$ small)
\begin{equation}\label{monformss}
(1+\omega(r))\Phi_r(x_{m,k},w_k)+ C \omega(r) \geq c_n \; \alpha_{m,k}^2\beta_{m,k}^2.
\end{equation}

We remark that the $w_k^\pm$ satisfy the assumptions of Proposition \ref{variant}. Indeed the $w_k^+$ are equilipschitz. To obtain a uniform modulus of continuity for the $w_k^-$ notice that, the $F(w_k^-)$ have an exterior tangent ball at $x_{m,k}$ of size $1/m$. Thus in a neighborhood of $x_{m,k}$ of size $2/m$ a modulus of continuity independent of $k$ can be obtained building an appropriate barrier. Outside such a neighborhood, the $w_k^-$ inherit the modulus of continuity of the $u^-$, because $w_k^-$ converges to $u^-$ uniformly.

Now from \eqref{phi0}, we also have that
\begin{equation}\label{lb}
\Phi_r(x_0, u) \geq c_n \alpha^2\beta^2 +o(1) \quad \text{as $r \to 0^+$.}
\end{equation}

On the other hand, since $u^\pm$ are Lipschitz continuous, for $\delta$ small and $r$ small depending on $\delta$
$$\int_{B_r(x_0)} |\nabla u^+|^2 d x = \int_{B_1} |\nabla u_r^+|^2 d x \leq \alpha^2 |B_1 \cap \{x\cdot \nu > \delta\}| + O(\delta) + o(1)$$
Analogously,
$$\int_{B_r(x_0)} |\nabla u^-|^2 d x \leq \beta^2 |B_1 \cap \{x \cdot \nu > \delta\}| + O(\delta) + o(1).$$
By Remark \ref{Fubini},
$$\Phi_r(x_0, u)= \Phi_1(0, u_r) \leq \frac 1 4 \alpha^2\beta^2  |B_1 \cap \{x \cdot \nu > \delta\}|^2 + O(\delta) + o(1).$$

This, together with \eqref{lb} gives that
$$
\lim_{r\to 0^+}\Phi_r(x_{0},u)= c_n \alpha^2\beta^2.
$$

Moreover, since $x_{m,k}\to x_m$ and $w_k\to u$ uniformly, we get from Lemma \ref{sandro} that 
\begin{equation}
\lim_{k\to\infty}\Phi_r(x_{m,k},w_k)=\Phi_r(x_{m},u).
\end{equation}

In particular, it follows from  that 
for every $\epsilon>0$ there exist $r>0$ small, and $m, k$ large (all depending on $\eps$) such that
$$
\Phi_r(x_{m,k}, w_k) \leq  c_n \; \alpha^2\beta^2+\epsilon.
$$
Applying \eqref{monformss} and recalling that
$$
\liminf_{m,k\to \infty}\alpha_{m,k}\geq \alpha,
$$
it follows that $\underline{\beta}\leq \beta,$ because $\alpha>0$ (by non-degeneracy.)
\end{proof}

\section{The function $u$ is a subsolution}\label{subsolutionuf}

In this section we want to show that $u$ satisfies part b.(ii) in Definition \ref{visc}, that is 
if $x_0\in F(u)$ is a regular point from the left with touching ball $B\subset \Omega^-(u),$ then near to $x_0$
$$
u^-(x)=\beta\langle x-x_0,\nu\rangle^-+o(\mid x-x_0\mid),\quad \beta\geq 0,
$$ 
in $B,$ and
$$
u^+(x)=\alpha\langle x-x_0,\nu\rangle^++o(\mid x-x_0\mid),\quad \alpha \geq 0,
$$ 
in $\Omega\setminus B$ with $\alpha\geq G(\beta, x_0, \nu).$ 

Notice that, even if $\beta=0,$ then $\Omega^+(u)$ and $\Omega^-(u)$ are tangent to $\{\langle x-x_0,\nu\rangle=0\}$ at $x_0$ since $u^+$ is non-degenerate. 
Thus $u$ has a full asymptotic development as in the next lemma.

 \begin{lem}
 Assume that near $x_0 \in F(u)$,
 $$
 u(x)=\alpha\langle x-x_0,\nu\rangle^+-\beta\langle x-x_0,\nu\rangle^-+o(\mid x-x_0\mid),
 $$
 with $\alpha> 0,$ $\beta\geq 0.$ Then
 $$
 \alpha\geq G(\beta, x_0, \nu).
 $$
 
 \end{lem}
\begin{proof} Assume by contradiction that $\alpha < G(\beta,x_0,\nu)$. We will show that in this case we can build a supersolution $w \in \mathcal F$ which is strictly smaller than $u$ at some point, contradicting the minimality of $u$.
Let $u_0$ be the two-plane solution, i.e. 
$$
u_0(x):=\lim_{r \to 0}\frac{u(x_0+r x)}{r}=\alpha\langle x,\nu\rangle^+-\beta\langle x,\nu\rangle^-.
$$

Suppose that $\alpha\leq G(\beta, x_0, \nu)-\delta_0$ with $\delta_0>0.$ Fix $\zeta=\zeta(\delta_0)$, to be made precise later. 

In view of Corollary \ref{unegsuicompact}, we can find $w_k \searrow u$ uniformly and for $r$ small, $k$ large
the rescaling $w_{k,r}$
satisfies  the following:
\begin{itemize}
\item[(i)] if $\beta>0,$ then $w_{k, r}(x)\leq u_0+\zeta\min\{\alpha,\beta\}$ on $\partial B_1,$
\item[(ii)] if $\beta=0,$ then $w_{k, r}(x)\leq u_0+\alpha\zeta$ on $\partial B_1,$
and 
$$
w_{k,r}(x)\leq 0, \quad 
\text{in $\{\langle x,\nu\rangle<-\zeta\}\cap\overline{B}_1.$}$$
\end{itemize}

In particular,
$$w_{k,r}(x) \leq u_0(x+\zeta \nu) \quad \text{on $\p B_1.$}$$

If $\beta >0$, let $v$ satisfy (using the notation in \eqref{notation})
\begin{equation}
\left\{\begin{array}{l}
\mathcal L_r v=r f_1^r ,\quad\{\langle x,\nu\rangle>-\zeta+\epsilon\phi(x)\}\\
\mathcal L_r v=r f_2^r ,\quad\{\langle x,\nu\rangle<-\zeta+\epsilon\phi(x)\}\\
v(x)=0,\quad \{\langle x,\nu\rangle=-\zeta+\epsilon\phi(x)\}\\
v(x)=u_0(x+\zeta\nu),\quad \partial B_1
\end{array}
\right.
\end{equation}
where $\phi\geq 0$ is a cut-off function, $\phi\equiv 0$ outside $B_{1/2},$ $\phi\equiv 1$ inside $B_{1/4}.$

For $\beta =0,$ replace the second equation with $v=0.$

Along the new free boundary, $F(v)=\{\langle x,\nu\rangle=-\zeta+\epsilon\phi(x)\}$ we have the following estimates:
$$|v_\nu^+ -\alpha| \leq c(\eps + \zeta) + Cr, \quad |v_\nu^- -\beta| \leq c(\eps + \zeta) + Cr,$$
with $c, C$ universal.

Indeed, 
$$
v^+-\alpha\langle x,\nu\rangle^+
$$
is solution of 
$$
\mathcal L_r  (v-\alpha\langle x,\nu\rangle^+)=g_r \quad g_r:=r \left(f_1^r -\alpha\mbox{div}(A_r \nu)\right).
$$
Thus, by standard $C^{1,\gamma}$ estimates
$$
|v^+_\nu-\alpha|\leq C\left(\|v-\alpha\langle x,\nu\rangle^+\|_\infty+[-\gamma+\epsilon\phi]_{1,\gamma}+ r\| f_1\|_\infty +r [A]_{0,\gamma}\right),
$$
which gives the desired bound. Similarly, one gets the bound for $v_\nu^-.$

Hence, since $\alpha \leq G(\beta, x_0, \nu(x_0)) - \delta_0$, say for $\eps=2\zeta$ and $\zeta, r$ small depending on $\delta_0$ $$v_\nu^+ < G(v_\nu^-, x_0, \nu),$$
and  the function, 
\begin{equation*}
\bar{w}_k=\left\{
\begin{array}{l}
\min\{w_k,\lambda v(\frac{x-x_0}{\lambda})\}\quad\mbox{in}\:\:B_{\lambda}(x_0),\\
w_k\quad\mbox{in}\:\:\Omega\setminus B_{\lambda}(x_0),
\end{array}
\right.
\end{equation*}
is still in $\mathcal{F}.$ 
However,  the set 
$$
\{\langle x,\nu\rangle\leq -\zeta+\epsilon\phi\}
$$  
contains a neighborhood of the origin, hence rescaling back $x_0 \in \Omega^-(\bar w_k)$. We get a contradiction since $x_0\in F(u)$ and
$\Omega^+(u)\subseteq\Omega^+(\bar{w}_k).$
\end{proof}
\section{The size of the reduced boundary}\label{sizereducedbound}

In this section we prove our regularity Theorem \ref{regularity}. First, we need the following standard result.

\begin{thm}\label{standard} Let $u$ be a solution to \eqref{FBintro}, such that $u$ is Lipschitz and non-degenerate.
Let $x_0 \in F(u) \cap B_{1}$ and $0<\epsilon<\delta<1.$ Then the following quantities are comparable with $\delta^{n-1}$:
\begin{itemize}
\item[(i)] $\frac{1}{\epsilon}\mid \{0<u<\epsilon\}\cap B_{\delta}(x_0)\mid,$
\item[(ii)] $ \frac{1}{\epsilon}\mid\mathcal{N}_{\epsilon}(F(u))\cap B_\delta(x_0)\mid,$
\item[(iii)] $N\epsilon^{n-1},$ where $N$ is the number of any family of balls of radius $\epsilon$ with finite overlapping covering $F(u)\cap B_\delta(x_0).$
\end{itemize} 
\end{thm}
\begin{proof} We follow the proof of Lemma 10 in \cite{C3}. It suffices to show that
\begin{equation}\label{1h} \int_{\{0< u<\eps\} \cap B_\delta(x_0)} |\nabla u|^2 \sim \eps \delta^{n-1},
\end{equation}
\begin{equation}\label{2h}
\int_{B_\eps(x_0)} |\nabla u|^2 \sim \eps^n.
\end{equation}
Then, the argument is the same as in the above cited lemma. We notice that the constants in these comparisons depend on the Lipschitz and non-degenerate bounds for $u$, say $C_1,c_1$. In what follows, dependance on $C_1, c_1$ (as well as the other universal parameters of the problem) is understood and constants depending on these parameters are still called universal.

Inequality \eqref{2h} follows from standard methods (using Poincare's inequality for the lower bound). Indeed, since $u^+$ is Lipschitz and non-degenerate$$
\sup_{B_\eps(x_0)} u^+ \sim \eps,
\quad 
\inf_{B_\eps(x_0)} u^+=0.$$

To prove \eqref{1h}, we rescale
$$u_\delta(x)= \frac{u(x_0 + \delta x)}{\delta}, \quad x \in B_1,$$ and use the notation in \eqref{notation}.
Let $u_{\eps, s} = \max(s/\delta, \min(u_\delta, \eps/\delta)), 0 < s < \eps.$
Then,
\begin{equation*}\label{equR1}
\begin{split}&-\delta \int_{B_1}f_1^\delta u_{\epsilon,s}=-\int_{B_1}u_{\epsilon,s}\mathcal L_\delta u^+_\delta\\
&=\int_{B_1}\langle A_\delta(x)\nabla u^+_\delta,\nabla u^+_{\epsilon,s}\rangle dx-\int_{\partial B_{1}}\langle A(x)\nabla u^+_\delta,\nu\rangle u_{\epsilon,s}d\mathcal{H}^{n-1}\\
&=\int_{B_1\cap\{0<s/\delta<u_\delta<\epsilon/\delta\} }\langle A_\delta(x)\nabla u_\delta,\nabla u_\delta\rangle dx-\int_{\partial B_{1}}\langle A_\delta(x)\nabla u^+_\delta,\nu\rangle u_{\epsilon,s}d\mathcal{H}^{n-1},
\end{split}
\end{equation*}
because $\nabla u_{\epsilon,s}=\nabla u_\delta \cdot\chi_{\{s/\delta<u_\delta<\epsilon/\delta\}}.$

Hence by ellipticity, using that $u^+$ is Lipschitz and $f_1$ is bounded we get  ($\delta <1$)
\begin{equation*}\label{equR3bis}
\begin{split}
\  \int_{B_1\cap\{0<s/\delta< u_\delta<\epsilon/\delta\} }\mid\nabla u_\delta \mid^2 dx
\leq C \frac \epsilon \delta ,\end{split}
\end{equation*}
with $C$ universal.
Letting $s\to 0$ and rescaling back, we obtain the upper bound in \eqref{1h}.

To obtain the lower bound, let $V$ be the solution to 
\begin{equation}
\left\{\begin{array}{l}
\mathcal L_\delta V=-\frac{\chi_{B_\sigma}}{\mid B_\sigma\mid},\quad \mbox{in}\quad B_1\\
V=0,\quad\mbox{on}\quad \partial B_1
\end{array}
\right.
\end{equation}
with $\sigma$   to be chosen later. By standard estimates, see for example \cite{GT}, $V\leq C(\sigma)$ and $- \langle A_\delta \nabla V,\nu\rangle \sim C^*$ on $\partial B_1.$ 
By Green formula ($u_\eps = u_{\eps,0}$)
\begin{equation}\label{equaR8}\begin{split}
\int_{B_1}(\mathcal L_\delta V) \frac{u_\delta^+ u_{\epsilon}}{\epsilon}-(\mathcal L_\delta \frac{u^+_\delta u_{\epsilon}}{\epsilon})V
=\int_{\partial B_1}\frac{u^+_\delta u_{\epsilon}}{\epsilon}\langle A_\delta \nabla V,\nu\rangle d\mathcal{H}^{n-1}
\end{split}
\end{equation}
because $V=0$ on $\partial B_1.$
We estimate
\begin{equation}\label{equaR9}
\delta \mid\int_{B_1}(\mathcal L_\delta V) \frac{u_\delta^+ u_{\epsilon}}{\epsilon}dx\mid=\mid \fint_{B_\sigma}\frac{u^+_\delta u_{\epsilon}}{\epsilon}dx\mid  \leq \bar C \sigma,
\end{equation}
because $u$ is Lipschitz, $0\leq u_\epsilon\leq \epsilon/\delta$.
From (\ref{equaR8}) and (\ref{equaR9}) and the fact that   $\langle A_\delta \nabla V,\nu\rangle \sim  - C^*$ on $\partial B_1$ we deduce that
\begin{equation*}
\label{equaR10}
\begin{split}
\delta \int_{B_1}(\mathcal L_\delta \frac{u^+_\delta  u_{\epsilon}}{\epsilon})Vdx &
\geq  -\bar C\sigma - \delta \int_{\partial B_1} \frac{u^+_\delta u_{\epsilon}}{\epsilon} \langle A_\delta\nabla V,\nu\rangle d\mathcal{H}^{n-1}\\
&\geq - \bar C \sigma + C^*\int_{\partial B_1} \frac{u^+_\delta u_{\epsilon}}{\epsilon} d\mathcal{H}^{n-1}. \end{split}
\end{equation*}
Thus using that $u^+$ is non-degenerate and choosing $\sigma$ small enough we get that
\begin{equation}
\label{equaR11}
\delta \int_{B_1}(\mathcal L_\delta \frac{u^+_\delta  u_{\epsilon}}{\epsilon})Vdx 
\geq  \tilde C.
\end{equation}
On the other hand in $\{0<u^+_\delta <\epsilon/\delta\}$
\begin{equation}
\label{equaR12}
\mathcal L_\delta( \frac{u^+_\delta u_\epsilon}{\epsilon})=\frac{2 \delta }{\epsilon}u_\eps f_1^\delta+\frac{1}{\epsilon}\langle A_\delta \nabla u_\delta,\nabla u_\delta \rangle.
\end{equation}

Combining \eqref{equaR11}-\eqref{equaR12} and using the ellipticity of $A_\delta$ we get that 
$$\frac{2 \delta^2 }{\epsilon}\int_{B_1} u_\eps f_1^\delta V+\frac{\delta \Lambda}{\epsilon}\int_{B_1} |\nabla u_\delta|^2 V \geq \bar C.$$

From the estimate  on $V$ we obtain that for $\delta$ small enough
$$\frac{\delta}{\epsilon}\int_{B_1} |\nabla u_\delta|^2 V \geq C$$
for some $C$ universal. Rescaling, we obtain the desired lower bound. 
\end{proof}

Let $u$ be the minimal solution constructed in Theorem \ref{teoprandbach}. Then, Theorem \ref{standard} above implies that $\Omega^+(u) \cap B_r(x), x \in F(u)$ is a set of finite perimeter. Next we show that in fact this perimeter is equivalent to $r^{n-1},$ and thus conclude the proof of Theorem \ref{regularity}. Constant depending possibly on the Lipschitz and non-degeneracy bounds for $u$ are still called universal.

\begin{thm}
Let $u$ be the minimal solution in Theorem  $\ref{teoprandbach}$. Then, the reduced boundary of $\Omega^+(u)$ has positive density in $\mathcal{H}^{n-1}$ measure at any point of $F(u)$, i.e. for $r<r_0$, $r_0$ universal 
$$
\mathcal{H}^{n-1}(F^*(u)\cap B_r(x))\geq c r^{n-1},
$$ for every $x \in F(u).$
\end{thm}
\begin{proof}The proof follows the lines of Corollary 4 in \cite{C3}. 
Let $w_k \in  \mathcal F$, $w_k \searrow u$ in $\bar B_1$ and $\mathcal L w_k = f_1$ in $\Omega^+(u).$ Let $x_0 \in F(u).$ As usual, we rescale and use the notation in \eqref{notation}:
$$u_r(x)= \frac{u(x_0 + r x)}{r}, \quad w_{k,r}= \frac{w_k(x_0 + r x)}{r} \quad x \in B_1.$$
As in Theorem \ref{standard}, we use the auxiliary function $V$ such that
\begin{equation}
\left\{\begin{array}{l}
\mathcal L_r V=-\frac{\chi_{B_\sigma}}{\mid B_\sigma\mid},\quad \mbox{in}\quad B_1\\
V=0,\quad\mbox{on}\quad \partial B_1. 
\end{array}
\right.
\end{equation}
Since $\nabla w_{k,r}$ is a continuous vector field in $\overline{\Omega_r^+(u_r) \cap B_1},$ we can use it to test for perimeter. 
Denoting for simplicity  $w_{k,r} = w$, we get 
\begin{equation}\label{equaF2}\begin{split}
&\int_{B_1\cap\Omega_r^+(u_r)}\left(V \mathcal L_r w -w\mathcal L_r V\right)\\
&=\int_{F^*(u_r)\cap B_1}\left(V\langle A_r\nabla w,\nu\rangle-w\langle A_r\nabla V,\nu\rangle\right)d\mathcal{H}^{n-1}-\int_{\partial B_1\cap \Omega_r^+(u_r)}w\langle A_r\nabla V,\nu\rangle d\mathcal{H}^{n-1}.
\end{split}
\end{equation}
Using estimates  for $V$ and the fact that the $w_k$ are uniformly Lipschitz, we get that 
\begin{equation}
\mid\int_{F^*(u_r)\cap B_1}V\langle A_r\nabla w,\nu\rangle d\mathcal{H}^{n-1}\mid\leq C(\sigma)\mathcal{H}^{n-1}(F^*(u_r)\cap B_1).
\end{equation}
As in \cite{C3} we have, as $k \to \infty$
$$
\int_{F^*(u_r)\cap B_1} w\langle A_r\nabla V, \nu\rangle d\mathcal{H}^{n-1} \to 0,
$$
$$
\int_{\partial B_1\cap\Omega_r^+(u_r)} w \langle A_r\nabla V,\nu\rangle d\mathcal{H}^{n-1} \to \int_{\partial B_1}u_r^+\langle A_r\nabla V,\nu\rangle d\mathcal{H}^{n-1}
$$
and
\begin{equation*}
\begin{split}
-\int_{B_1\cap\Omega_r^+(u_r)} w\mathcal L_r V \to \fint_{B_\sigma}u_r^+.
\end{split}
\end{equation*}
Passing to the limit in 
(\ref{equaF2}) and using all of the above we get
\begin{equation}\label{equaF3}\begin{split}
&|r\int_{B_1\cap\Omega^+(u_r)}V f_1^r +\fint_{B_\sigma}u^+_r +\int_{\partial B_1}u_r^+\langle A_r \nabla V,\nu\rangle d\mathcal{H}^{n-1}| \\ & \leq C(\sigma)\mathcal{H}^{n-1}(F^*(u_r)\cap B_1).
\end{split}
\end{equation}
Since $u$ is Lipschitz and non-degenerate, for $\sigma$ small
$$
\fint_{B_\sigma}u_r^+ \leq \bar C\sigma
$$
and using the estimate for $\langle A_r \nabla V, \nu \rangle$
$$
-\int_{\partial B_1}u_r^+\langle A_r \nabla V,\nu\rangle d\mathcal{H}^{n-1}\geq \bar c>0.
$$
Also, since $f_1^r$ is bounded 
$$\int_{B_1\cap\Omega_r^+(u_r)}V f_1^r \leq \bar C(\sigma).$$

Hence choosing first $\sigma$ and then $r$ sufficiently small we get that the left-hand-side in equation \eqref{equaF3} is larger than a constant $\tilde C$, which concludes our proof.
\end{proof}


\begin{thebibliography}{9999}










\bibitem[C3]{C3} Caffarelli L.A., \emph{A Harnack inequality approach to the regularity of free boundaries. III. Existence theory, compactness, and dependence on X}, Ann. Scuola Norm. Sup. Pisa Cl. Sci. (4) \textbf{15} (1988), no. 4, 583Ð602 (1989).
\bibitem[CS]{CS} Caffarelli L. A., Salsa S., \emph{A Geometric Approach to Free Boundary Problems,} Graduate Studies in Mathematics, vol. 68 2005.


\bibitem[DFSs1]{DFSs1} De Silva D., Ferrari F., Salsa S., {\it On two phase free boundary problems governed by elliptic equations with distributed sources}, Discrete Contin. Dyn. Syst. Ser. S 7, no. 4, 673--693 (2014).

\bibitem[DFS4]{DFS4} De Silva D., Ferrari F., Salsa S., {\it The regularity of flat free boundaries for a non-homogeneous two-phase problem in divergence form}, in preparation.  




\bibitem[GT]{GT}  Gilbarg D., Trudinger N. S., {\it Elliptic partial differential equations of second order. } Classics in Mathematics. Springer-Verlag, Berlin, (2001). 

\bibitem[KS]{KS} Kinderlehrer D., Stampacchia G., {\it An introduction to variational inequalities and their applications}, Classics in Applied Mathematics, SIAM 2000.


\bibitem[LSW]{LSW} Littman W., Stampacchia G.., Weinberger H. F., {\it Regular points for elliptic equations with discontinuous coefficients.} Ann. Scuola Norm. Sup. Pisa 17,  43--77, (1963)

\bibitem[MP]{MP} Matevosyan N., Petrosyan A., {\it  Almost monotonicity formulas for elliptic and parabolic operators with variable coefficients}, Comm. Pure Appl. Math., 64(2):271-311, 2011.

\bibitem[T]{T} Troianiello G. M., {\it Elliptic Differential Equations and Obstacle Problems}, Springer Science \& Business Media, Jul 31, 1987.








\end{thebibliography}
\end{document}